\documentclass{amsart}
\usepackage[T1]{fontenc}
\usepackage[utf8]{inputenc}
\usepackage{amsmath}
\usepackage{amssymb}
\usepackage{caption}
\usepackage{amsthm}
\usepackage[hidelinks,colorlinks=true,citecolor=magenta,linkcolor=cyan]{hyperref}
\usepackage{cleveref}
\usepackage{hyperref}
\usepackage{nicefrac}
\usepackage[inline]{enumitem}
\usepackage[bb = fourier,cal = euler,scr = rsfs]{mathalfa}
\usepackage{mathabx}
\usepackage{enumitem}
\usepackage[svgnames,table]{xcolor}

\usepackage{amssymb}
\usepackage{mathrsfs}

\usepackage{array,float}

\usepackage{tikz}
\usetikzlibrary{shapes,arrows}
\usetikzlibrary{fit,positioning}
\usetikzlibrary{patterns,decorations.pathreplacing}
\usepackage{xkeyval}
\usepackage{moreverb}
\usepackage{epic}
\usepgfmodule{shapes,plot,decorations}
\usepackage{booktabs} 
\usepackage[normalem]{ulem}

\usepackage[warn]{mathtext}

\newtheorem{theorem}{Theorem}[section]
\newtheorem{corollary}[theorem]{Corollary}
\newtheorem{fact}[theorem]{Fact}
\newtheorem{proposition}[theorem]{Proposition}
\newtheorem{claim}{Claim}

\newtheorem{question}[theorem]{Question}

\theoremstyle{remark}
\newtheorem{remark}[theorem]{Remark}

\theoremstyle{remark}
\newtheorem{example}[theorem]{Example}

\title{On topologically transitive subsets of flag manifolds}

\author{Subhadip Dey}
\address{Tata Institute of Fundamental Research, School of Mathematics, Homi Bhabha Road, Colaba, Mumbai 400005, India}
\email{subhadip@math.tifr.res.in}

\author{Sami Douba}
\address{Mathematisches Institut der Universit\"at Bonn, Endenicher Allee 60, 53115 Bonn, Germany}
\email{douba@math.uni-bonn.de}

\author{Konstantinos Tsouvalas}
\address{Max Planck Institute for Mathematics in the Sciences, Inselstraße 22, 04103 Leipzig, Germany}
\email{konstantinos.tsouvalas@mis.mpg.de}

\usepackage{chngcntr}
\usepackage{graphicx} 
\usepackage{float}
\counterwithout{equation}{section}
\counterwithout{theorem}{section}

\begin{document}

\begin{abstract}
We discuss some contexts in which the topological dynamics of certain Anosov subgroups of higher-rank Lie groups on ``bad'' subsets of flag manifolds can be analyzed using results from homogeneous dynamics in the infinite-covolume rank-one setting. For example, we show that a group of projective transformations dividing a strictly convex domain in projective space and intersecting Zariski-densely the stabilizer of an ellipsoid has a dense orbit in the complement of the domain, and acts minimally on the space of full projective flags tangent to the domain. This accounts for all known examples of divisible strictly convex domains in sufficiently high dimensions. We also provide examples of Zariski-dense groups that are Anosov in a partial flag manifold but such that the equivariant projection from the Benoist--Guivarc'h limit set in the Furstenberg boundary to the Anosov limit set is not a fibration.
\end{abstract}

\maketitle

A nonempty open subset $\Omega$ of the $n$-dimensional projective space $\mathbb{RP}^n$ is said to be {\em properly convex} if the closure $\overline{\Omega}$ is contained in an affine chart of $\mathbb{RP}^n$. If $\overline{\Omega}$ is moreover strictly convex when viewed in that affine chart, one says $\Omega$ is {\em strictly convex}.
Following Benz\'ecri~\cite{MR124005}, a properly convex domain $\Omega \subset \mathbb{RP}^n$ is {\em divisible} if $\Omega$ is invariant under a subgroup $\Gamma < \mathrm{SL}_{n+1}(\mathbb{R})$ that acts properly and cocompactly on~$\Omega$, in which case $\Gamma$ is said to {\em divide} $\Omega$. It was proved by Benoist~\cite{MR2094116} that a divisible properly convex domain $\Omega$ is strictly convex if and only if $\partial \Omega$ has~$\mathcal{C}^1$ regularity, and that each of the latter two conditions is in turn equivalent to Gromov-hyperbolicity of any dividing group $\Gamma$. The action of $\Gamma$ on $\partial \Omega$ is then topologically conjugate to that of $\Gamma$ on the abstract Gromov boundary of $\Gamma$, and is in particular minimal. The present work is motivated by questions regarding the topological dynamics of $\Gamma$ on some other spaces that are naturally associated to $\Omega$.

\begin{question}\label{domain}
Suppose $\Gamma < \mathrm{SL}_{n+1}(\mathbb{R})$ divides a strictly convex domain $\Omega \subset \mathbb{RP}^n$. 
\begin{enumerate}
\item\label{denseorbitdomain} Does $\Gamma$ have a dense orbit in $\mathbb{RP}^n - \overline{\Omega}$? 
\item\label{minimalitydomain} Does $\Gamma$ act minimally on the space of all complete flags of projective subspaces $\{x\} = V_0 \subset V_1 \subset \cdots \subset V_{n-1}\subset \mathbb{RP}^n$ such that $x \in \partial \Omega$ and the $V_i$ are all tangent to $\partial\Omega$?
\end{enumerate}
\end{question}

In the case that $\Omega$ is an ellipsoid, it follows from Moore's ergodicity theorem~\cite{zbMATH03239991} that the action of $\Gamma$ on each of the spaces in Question~\ref{domain} is ergodic with respect to the Lebesgue measure class and therefore possesses a dense orbit.\footnote{Note that, even when $\Omega$ is an ellipsoid, the action of $\Gamma$ on $\mathbb{RP}^n - \overline{\Omega}$ need not be minimal, and indeed, it follows from work of Ratner~\cite{MR1106945} and Shah~\cite{zbMATH00849271} that for $n\geq 3$ minimality is, in the ellipsoid case, equivalent to the absence of compact immersed totally geodesic hypersurfaces within the hyperbolic orbifold $\Gamma \backslash \Omega$.} Ergodicity can also be used to show that the action in Question~\ref{domain}(\ref{minimalitydomain}) is moreover minimal in this case; see~\cite[Lemma~8.5]{zbMATH03418289}. In the sequel (Corollary~\ref{domainexample}), we will answer Question~\ref{domain} affirmatively in the particular case that there is a stabilizer  $H$ in $\mathrm{SL}_{n+1}(\mathbb{R})$ of some ellipsoid in $\mathbb{RP}^n$ such that $\Gamma \cap H$ is Zariski-dense in $H$. 

Note that if $n\geq 3$ and $\Omega$ as in Question~\ref{domain} is not an ellipsoid, then even if the dividing group~$\Gamma$ is abstractly isomorphic to a group dividing an ellipsoid $\Omega_0 \subset \mathbb{RP}^n$, there will be no topological conjugacy\footnote{Such a topological conjugacy indeed exists for $n=2$ since in this case the exterior of the domain can be identified with the space of unordered pairs of distinct points in the Gromov boundary of the dividing group.} between the induced action of $\Gamma$ on $\mathbb{RP}^n - \overline{\Omega_0}$ and the original action on $\mathbb{RP}^n - \overline{\Omega}$. The lack of such a conjugacy can easily be seen as follows: by a result of Benoist~\cite[Th\'eor\`eme~3.6]{Benoist-cones}, a subgroup of $\mathrm{SL}_{n+1}(\mathbb{R})$ dividing a strictly convex domain in $\mathbb{RP}^n$ is Zariski-dense in $\mathrm{SL}_{n+1}(\mathbb{R})$ as soon as the domain is not an ellipsoid. Now, by the appendix of~\cite{MR1198811}, every Zariski-dense subgroup of $\mathrm{SL}_{n+1}(\mathbb{R})$ contains elements all of whose nontrivial powers have precisely $n+1$ fixed points in $\mathbb{RP}^n$. On the other hand, there is no such element of $\mathrm{SL}_{n+1}(\mathbb{R})$ preserving an ellipsoid in $\mathbb{RP}^n$ if $n\geq 3$. We remark moreover that, as demonstrated by Benoist~\cite[Proposition~3.1]{MR2295544} for $n=4$ and by Kapovich~\cite{zbMATH05220954} for arbitrary $n \geq 4$, there are subgroups of $\mathrm{SL}_{n+1}(\mathbb{R})$ dividing strictly convex domains in $\mathbb{RP}^n$ that are not abstractly commensurable to groups dividing ellipsoids.

 We may recast Question~\ref{domain} in the following more general framework using language of Kapovich--Leeb--Porti~\cite{MR3720343}.
 Let $G$ be a noncompact semisimple real algebraic group. Given a proper parabolic subgroup $Q < G$, a reflexive proper parabolic subgroup $Q' < G$, a closed subset $\Lambda \subset G/Q'$, and an element $w \in Q'\backslash G/Q$, denote by $w(\Lambda) \subset G/Q$ the set of all $\varphi \in G/Q$ such that $\varphi$ is in position $w$ relative to some point of $\Lambda$. If $\Lambda$ is $Q'$-antipodal and $w$ is contained in a slim ideal of $Q'\backslash G/Q$, then the map $w(\Lambda) \rightarrow \Lambda$ that assigns to each $\varphi \in w(\Lambda)$ the unique point in $\Lambda$ relative to which $\varphi$ is in position $w$ is a continuous fibration; see~\cite[Lemma~7.4]{MR3720343}.

Now suppose we are given a nontrivial representation $\tau: \mathrm{Spin}_\circ(n,1) \rightarrow G$, $n \geq 2$, and let $H$ be the image of $\tau$ in $G$. Let $P_H=M_H A_H N_H$ be a Langlands decomposition of a proper parabolic subgroup $P_H < H$. Let $Q_\tau < G$ be a reflexive proper parabolic subgroup of $G$ that is compatible with $\tau$, so that $H \cap Q_\tau$ is a proper parabolic subgroup of $H$. Given a closed subgroup $\Gamma < G$, denote by~$\Lambda_\Gamma^\tau$ the flag limit set of $\Gamma$ in $G/Q_\tau$. If $\Lambda_\Gamma^\tau$ is $Q_\tau$-antipodal and $w \in Q_\tau \backslash G /Q$ is contained in a slim ideal of $Q_\tau \backslash G/Q$, then the fibration $w(\Lambda_\Gamma^\tau) \rightarrow \Lambda_\Gamma^\tau$ described in the previous paragraph is $\Gamma$-equivariant. Note that, under the above assumptions, we have that~$\Lambda_H^\tau$ is nothing but the image of $H$ in $G/Q_\tau$ under the quotient map $G \rightarrow G/Q_\tau$.

The following statement gives sufficient conditions for the action of a discrete subgroup $\Gamma < G$ on $w(\Lambda_\Gamma^\tau)$ to possess a dense orbit (respectively, to be minimal).

\begin{theorem}\label{thm:main}
Suppose $w \in Q_\tau \backslash G /Q$ is contained in a slim ideal of $Q_\tau \backslash G/Q$, and that $H$ acts transitively on $w(\Lambda_H^\tau)$. Let $\Gamma$ be a $Q_\tau$-regular antipodal\footnote{Such subgroups also appear in work of Canary--Zhang--Zimmer~\cite{zbMATH07924657} as {\em $Q_\tau$-transverse} subgroups.} subgroup of~$G$ such that $\Gamma \cap H$ is Zariski-dense in $H$.
\begin{enumerate}
\item\label{denseorbit} If a point stabilizer for the action of $H$ on $w(\Lambda_H^\tau)$ contains $A_H U_H$ for some unipotent one-parameter subgroup $U_H < N_H$, then $\Gamma$ has a dense orbit in~$w(\Lambda_\Gamma^\tau)$.
\item\label{minimality} If a point stabilizer for the action of $H$ on $w(\Lambda_H^\tau)$ contains $A_H N_H$, then the action of $\Gamma$ on $w(\Lambda_\Gamma^\tau)$ is minimal. 
\end{enumerate}
\end{theorem}

Examples of $Q_\tau$-regular antipodal subgroups of $G$ are the $Q_\tau$-Anosov subgroups in the sense of Labourie~\cite{MR2221137} and Guichard--Wienhard~\cite{MR2981818}. For example, sufficiently small deformations in $G$ of convex cocompact subgroups of $H$ are $Q_\tau$-Anosov; see \cite[Proposition~2.1]{MR2221137} and \cite[Theorem~1.2]{MR2981818}. If $\Gamma$ is $Q_\tau$-Anosov, then~$\Gamma$ is intrinsically Gromov-hyperbolic and $\Lambda_\Gamma^\tau$ is $\Gamma$-equivariantly homeomorphic to the Gromov boundary of $\Gamma$; depending on one's definition, these properties either feature in the definition of a $Q_\tau$-Anosov representation, or follow from~\cite[Theorem~1.4]{MR3890767}.

Note that, in the setting of Theorem~\ref{thm:main}, the condition that a point stabilizer for the action of $H$ on $w(\Lambda_H^\tau)$ contain $A_H N_H$ holds if and only if $w(\Lambda_H^\tau)$, or equivalently~$w(\Lambda_\Gamma^\tau)$, is compact. The latter in turn holds if and only if $w$ is the minimal element of $Q_\tau \backslash G / Q$. In this case, if $\Gamma$ is moreover Zariski-dense in $G$, then it follows from Theorem~\ref{thm:main}(\ref{minimality}) that $w(\Lambda_\Gamma^\tau)$ is precisely the limit set of $\Gamma$ in $G/Q$ in the sense of Guivarc'h~\cite{MR1074315} and Benoist~\cite{MR1437472}.

One immediately deduces from Theorem~\ref{thm:main} the following.

\begin{corollary}\label{domainexample}
Let $n \geq 3$. Suppose $\Gamma < \mathrm{SL}_{n+1}(\mathbb{R})$ divides a strictly convex domain $\Omega \subset \mathbb{RP}^n$, and that there is some conjugate $H$ of  $\mathrm{SO}_\circ(n,1)$ in $\mathrm{SL}_{n+1}(\mathbb{R})$ such that $\Gamma \cap H$ is Zariski-dense in $H$. For $0 \leq i \leq n-1$, let $\mathcal{E}_i$ be the space of all complete flags of projective subspaces $V_0 \subset V_1 \subset \cdots \subset V_{n-1} \subset \mathbb{RP}^n$ such that $V_j \cap \overline{\Omega} = \emptyset$ for $j < i$ and $V_j \cap \overline{\Omega}$ is a singleton for $j \geq i$.  
\begin{enumerate}
\item\label{cor:denseorbitdomain} If $0 \leq i \leq n-2$, then $\Gamma$ has a dense orbit in $\mathcal{E}_j$. In particular, there is a dense $\Gamma$-orbit in $\mathbb{RP}^n - \overline{\Omega}$. 
\item\label{cor:minimalitydomain} The action of $\Gamma$ on $\mathcal{E}_0$, that is, on the space of complete projective flags tangent to $\partial\Omega$, is minimal.
\end{enumerate}
\end{corollary}

Examples of divisible $\Omega \subset \mathbb{RP}^n$ as in Corollary~\ref{domainexample} include the developing image of any convex projective structure on a closed hyperbolic $n$-manifold obtained by ``bending'' the hyperbolic structure along a closed embedded two-sided totally geodesic hypersurface (see~\cite[\S1.3]{MR2648674} and the references therein), as well as the aforementioned examples~\cite{MR2295544, zbMATH05220954}, which account for all known examples, of divisible strictly convex~$\Omega$ with dividing groups that are not quasiisometric to~$\mathbb{H}^n$.

Theorem~\ref{thm:main} also has the following consequence. Following~\cite{zbMATH06088868}, for $n \geq 2$, one says a discrete subgroup $\Gamma < \mathrm{SO}(n,2)$ is {\em $\mathrm{AdS}$-quasiFuchsian} if there is a (necessarily unique) $\Gamma$-invariant acausal topological $(n-1)$-sphere $\Lambda_\Gamma^\tau$ in the Einstein universe $\mathrm{Ein}^n$ such that the action of $\Gamma$ on the convex hull of $\Lambda_\Gamma^\tau$ in $\mathrm{AdS}^{n,1}$ is cocompact.\footnote{Note that, unlike in~\cite{zbMATH06088868}, we do not assume here that $\Gamma$ is commensurable to the fundamental group of a closed real hyperbolic manifold.} By~\cite[Theorem~1.24~and~Corollary 1.26]{MR4862504}, when $n \geq 3$, a subgroup $\Gamma < \mathrm{SO}(n,2)$ is $\mathrm{AdS}$-quasiFuchsian if and only if $\Gamma$ is Gromov-hyperbolic with Gromov boundary an $(n-1)$-sphere and is $Q_\tau$-Anosov in $\mathrm{SO}(n,2)$, where $Q_\tau$ denotes the stabilizer in $\mathrm{SO}(n,2)$ of an isotropic line in $\mathbb{R}^{n,2}$. In the latter case, the invariant acausal sphere $\Lambda_\Gamma^\tau \subset \mathrm{Ein}^n$ is precisely the limit set of $\Gamma$ in $\mathrm{Ein}^n = G/Q_\tau$. 

\begin{corollary}\label{sopq}
Let $n \geq 3$, and suppose  $\Gamma < \mathrm{SO}(n,2)$ is $\mathrm{AdS}$-quasiFuchsian. If there is some conjugate~$H$ of $\mathrm{SO}_\circ(n,1)$ in $\mathrm{SO}(n,2)$ such that $\Gamma \cap H$ is Zariski-dense in~$H$, then
\begin{enumerate}
\item\label{cor:denseorbitsopq} there is a dense $\Gamma$-orbit in the Furstenberg boundary $\mathcal{F}_{\mathrm{SO}(n,2)}$ of $\mathrm{SO}(n,2)$, that is, in the space of complete flags of isotropic subspaces of $\mathbb{R}^{n,2}$; and
\item\label{cor:minimalitysopq} the action of $\Gamma$ on the space of complete flags of isotropic subspaces of $\mathbb{R}^{n,2}$ containing an isotropic line in $\Lambda_\Gamma^\tau$ is minimal. In particular, the action of~$\Gamma$ on the space of photons in~$\mathrm{Ein}^n$ is minimal.
\end{enumerate}
\end{corollary}

A crucial observation in~\cite{arXiv:2608.27274} is that there is an open dense $\mathrm{SO}_{\circ}(n,1)$-orbit in $\mathcal{F}_{\mathrm{SO}(n,2)}$ with noncompact point stabilizers. By Moore's ergodicity theorem, this implies that any lattice in $\mathrm{SO}_{\circ}(n,1)$ acts ergodically on $\mathcal{F}_{\mathrm{SO}(n,2)}$ with respect to the Lebesgue measure class. It was further shown in~\cite{arXiv:2608.27274} that, for $n\geq3$, if $\Gamma < \mathrm{SO}_{\circ}(n,1)$ is a cocompact lattice, then the action on $\mathcal{F}_{\mathrm{SO}(n,2)}$ of any sufficiently small deformation of $\Gamma$ within $\mathrm{SO}(n,2)$ remains ergodic, and in particular continues to possess a dense orbit in $\mathcal{F}_{\mathrm{SO}(n,2)}$.
However, while examples of $\Gamma$ as in Corollary~\ref{sopq} include all $\Gamma$ obtained by ``bending''  cocompact lattices in $\mathrm{SO}_\circ(n,1)$ (see~\cite[\S6]{MR2915550}), they also include all known examples~\cite{MR3968766, MR5111594} of $\mathrm{AdS}$-quasiFuchsian subgroups of~$\mathrm{SO}(n,2)$ that are {\em not} virtually isomorphic to uniform lattices in $\mathrm{SO}_\circ(n,1)$.

We also record the following instance of Theorem~\ref{thm:main}(\ref{minimality}). Given an integer $d \geq 2$ and a sequence of integers $1 \leq d_1 < \ldots < d_k \leq d-1$, denote by $\mathcal{F}_{d_1, \ldots, d_k}(\mathbb{R}^d)$ the space of flags $V_{d_1} \subset \cdots \subset V_{d_k}$ of linear subspaces of $\mathbb{R}^d$ of dimension $d_1, \ldots, d_k$, respectively, and let 
$P_{d_1, \ldots, d_k}$ be the stabilizer in $\mathrm{SL}_d(\mathbb{R})$ of such a flag, so that $\mathcal{F}_{d_1, \ldots, d_k}(\mathbb{R}^d) \simeq \mathrm{SL}_d(\mathbb{R})/P_{d_1, \ldots, d_k}$ as $\mathrm{SL}_d(\mathbb{R})$-spaces. We will denote by $\mathcal{F}(\mathbb{R}^d)$ the space $\mathcal{F}_{1, \ldots, d-1}(\mathbb{R}^d)$ of complete flags of linear subspaces of $\mathbb{R}^d$. Given a discrete subgroup $\Gamma < \mathrm{SL}_d(\mathbb{R})$, denote by $\Lambda_\Gamma^{d_1, \ldots, d_k}$ (respectively, by $\Lambda_\Gamma^\mathcal{F}$) the flag limit set of~$\Gamma$ in $\mathcal{F}_{d_1, \ldots, d_k}(\mathbb{R}^d)$ (respectively, in $\mathcal{F}(\mathbb{R}^d)$).

\begin{corollary}\label{rp3}
Let $\tau: \mathrm{SL}_2(\mathbb{C}) \rightarrow \mathrm{SL}_4(\mathbb{R})$ be the representation given by restricting scalars, and let $\Gamma$ be a $P_2$-Anosov subgroup of $\mathrm{SL}_4(\mathbb{R})$ such that the limit set $\Lambda_\Gamma^2$ of~$\Gamma$ in the $2$-Grassmannian $\mathcal{F}_2(\mathbb{R}^4)$ is a topological $2$-sphere. If there is a conjugate~$H$ of $\tau(\mathrm{SL}_2(\mathbb{C}))$ in $\mathrm{SL}_4(\mathbb{R})$ such that $\Gamma \cap H$ is Zariski-dense in $H$, then the action of~$\Gamma$ on $\mathbb{RP}^3$ is minimal.
\end{corollary}

The following statement, while not an immediate consequence of Theorem~\ref{thm:main}, is proved in a similar vein.

\begin{proposition}\label{sl3r}
Let $\Gamma$ be a regular antipodal subgroup of $\mathrm{SL}_3(\mathbb{R})$ whose limit set is a topological circle. If $\Gamma \cap H$ is Zariski-dense in $H$ for some reducible copy $H$ of $\mathrm{SL}_2(\mathbb{R})$ in $\mathrm{SL}_3(\mathbb{R})$, then $\Gamma$ has a dense orbit in $\mathbb{RP}^2$.
\end{proposition}

Given a semisimple real algebraic group $G$, a minimal parabolic subgroup $B < G$, a proper parabolic subgroup $P < G$, and a Zariski-dense discrete subgroup $\Gamma < G$, work of Benoist~\cite{MR1770716} implies surjectivity of the $\Gamma$-equivariant map $\Lambda_\Gamma^B \rightarrow \Lambda_\Gamma^P$ between the Benoist--Guivarc'h limit sets $\Lambda_\Gamma^B$ and $\Lambda_\Gamma^P$ of $\Gamma$ in $G/B$ and $G/P$, respectively. In the case that $\Gamma$ is moreover $P$-Anosov, one might then be led to believe, based on the cases where Theorem~\ref{thm:main}(\ref{minimality}) applies with $B = Q$ and $P=Q_\tau$, or based on the examples where $\Gamma$ is in fact $B$-Anosov, that this map is always a fibration. One cannot expect the latter to hold in general, as demonstrated by the following examples.

\begin{theorem}\label{distinctfibers} For any integer $n \geq 2$, there exists a Zariski-dense $P_1$-Anosov rank-two free subgroup $\Gamma<\mathrm{SL}_{2n}(\mathbb{R})$ such that the $\Gamma$-equivariant projection $\Lambda_{\Gamma}^{\mathcal{F}}\rightarrow \Lambda_{\Gamma}^{1,2n-1}$ has a fiber that is a singleton and another fiber that is infinite. \end{theorem}

\section{Proof of Theorem~\ref{thm:main}}

To prove Theorem~\ref{thm:main}, we will reduce to known results regarding the dynamics of Zariski-dense groups of conformal transformations of the sphere on certain bundles over the sphere. 
For $d\geq 2$, let $\mathbb{S}^d$ denote the conformal $d$-sphere. We define a {\em double-pointed complete oriented flag of spheres} in $\mathbb{S}^d$ to be a sequence $S_0 \subset S_1 \subset \cdots \subset S_{d-1} \subset \mathbb{S}^d$ of nested round spheres in $\mathbb{S}^d$ with $\dim(S_i) = i$, together with a choice of orientation for each of the~$S_i$; we will refer to the two elements of~$S_0$ as the {\em marked points} of the flag $(S_0, \ldots, S_{d-1})$. We define a {\em pointed complete oriented flag of spheres} in $\mathbb{S}^d$ similarly, but the subspace $S_0$ is replaced with a single point of $\mathbb{S}^d$ (as opposed to an ordered pair of distinct points in $\mathbb{S}^d$); we will refer to this point as the {\em marked point} of the given flag. Let $\ddot{\mathcal{F}}_d$ (respectively,~$\dot{\mathcal{F}}_d$) denote the space of all double-pointed complete oriented flags of spheres (resp., the space of all {\em pointed} complete oriented flags of spheres) in $\mathbb{S}^d$. If $P=MAN$ is a Langlands decomposition of a proper parabolic subgroup $P$ of the group $H_d$ of orientation-preserving conformal transformations of~$\mathbb{S}^d$, then $\ddot{\mathcal{F}}_d$ (respectively,~$\dot{\mathcal{F}}_d$) can be identified as a homogeneous $H_d$-space with $H_d/A$ (resp., with $H_d/(AU)$, where $U$ is any unipotent one-parameter subgroup of $N$). Under this interpretation, the ``forgetful maps'' $\ddot{\mathcal{F}}_d \rightarrow \dot{\mathcal{F}}_d$ and $\dot{\mathcal{F}}_d \rightarrow \mathbb{S}^d$ are nothing but the quotient maps $H_d/A \rightarrow H_d/(AU)$ and $H_d/(AU) \rightarrow H_d/P$, respectively.

Given a discrete subgroup $\Delta < H_d$ with limit set $\Lambda_\Delta \subset \mathbb{S}^d$, we will denote by~$\dot{\mathcal{F}}_d(\Delta)$ (respectively, by $\ddot{\mathcal{F}}_d(\Delta)$) the set of all pointed complete oriented flags of spheres in $\mathbb{S}^d$ whose marked point lies in $\Lambda_\Delta$ (resp., the set of all double-pointed complete oriented flags of spheres in $\mathbb{S}^d$ both of whose marked points lie in $\Lambda_\Delta$). We borrow the following argument from the proof of Theorem~3.1 in~\cite{MR4953215}; see also~\cite[Corollary~4.5]{MR3674219}. We repeat the argument here for the convenience of the reader.

\begin{proposition}\label{prop:spheres}
Let $\Delta$ be a Zariski-dense discrete subgroup of $H_d$. Then the image of $\ddot{\mathcal{F}}_d(\Delta)$ under the forgetful map $\ddot{\mathcal{F}}_d \rightarrow \dot{\mathcal{F}}_d$ is dense in $\dot{\mathcal{F}}_d(\Delta)$.
\end{proposition}

\begin{proof}
Denote by $\widetilde{\Lambda}_\Delta$ the preimage of $\Lambda_\Delta$ in the unit tangent bundle $T^1\mathbb{S}^d$ of $\mathbb{S}^d$. Given three distinct points $x, y, z \in \mathbb{S}^d$, denote by $\overline{xy}^z$ the circular arc from $x$ to $y$ excluding $z$ contained in the unique round circle through $x$, $y$, and $z$. Following~\cite{MR4953215}, given a unit tangent vector $v \in T^1\mathbb{S}^d$ based at a point $x \in \mathbb{S}^d$ and a sequence $y_j \in \mathbb{S}^d$ converging to $v$, we say $y_j$ {\em converges to $v$ tangentially} if, for some $z \in \mathbb{S}^d\smallsetminus \{x\}$, the unit tangent vector based at $x$ determined by the circular arc $\overline{xy_j}^z$ converges to $v$ as $j \rightarrow \infty$, in which case this convergence will indeed hold for any choice of $z \in \mathbb{S}^d \smallsetminus \{x\}$. Let $\mathcal{E}_\Delta$ denote the subset of $\widetilde{\Lambda}_\Delta$ consisting of all unit tangent vectors $v \in \widetilde{\Lambda}_\Delta$ such that there is a sequence in $\Lambda_\Delta$ converging to $v$ tangentially. Since $\mathcal{E}_\Delta$ is nonempty (by a compactness argument) and $\Delta$-invariant, we have by a result of Guivarc'h--Raugi~\cite[Theorem~2]{MR2339285} that $\mathcal{E}_\Delta$ is dense in $\widetilde{\Lambda}_\Delta$. 

Now let $\dot{\phi} = (x, S_1, \dots, S_{d-1}) \in \dot{\mathcal{F}}_d(\Delta)$. We exhibit a sequence in the image of $\ddot{\mathcal{F}}_d(\Delta)$ converging to $\dot{\phi}$. Since $\dot{\phi}$ is otherwise contained in this image, we may assume that the circle $S_1$ contains no point of $\Lambda_\Delta$ apart from $x$. Now let $z \in S_1 \smallsetminus \{x\}$, and let $v_1, \ldots, v_{d-1} \in T^1\mathbb{S}^d$ be a sequence of linearly independent unit tangent vectors based at $x$ that, together with the point $z$, determine the oriented flag of spheres $S_1 \subset \cdots \subset S_{d-1}$. Since $\mathcal{E}_\Delta$ is dense in $\widetilde{\Lambda}_\Delta$, we can find a sequence $u_k \in \mathcal{E}_\Delta$ converging to $v_1$ as $k \rightarrow \infty$. For each $k$, there is then a sequence $y_j^{(k)} \in \Lambda_\Delta$ converging to $u_k$ tangentially as $j \rightarrow \infty$. Denoting by $x_k \in \mathbb{S}^d$ the point at which the unit tangent vector $u_k$ is based, we can then choose $j_k$ large enough such that $y_{j_k}^{(k)} \rightarrow x$, and such that the unit tangent vector based at $x_k$ determined by the circular arc $\overline{x_ky_{j_k}^{(k)}}^z$ converges to $v_1$ as $k \rightarrow \infty$.

For $\ell = 2, \ldots, d-1$, we now choose a sequence $y_{\ell,k} \in \mathbb{S}^d$ converging to $x$ such that the unit tangent vector based at $x_k$ determined by the circular arc $\overline{x_ky_{\ell,k}}^z$ converges to $v_\ell$ as $k \rightarrow \infty$. For $k$ sufficiently large, we then have that $(x_k, y_{j_k}^{(k)}, y_{2,k}, \ldots, y_{d-1,k})$ is a $d$-tuple of pairwise distinct points of $\mathbb{S}^d$, and hence, together with the point~$z$, uniquely determines a pointed complete oriented flag of spheres $\dot{\phi}_k$ with marked point $x_k$. By construction, we then have that $\dot{\phi}_k \rightarrow \dot{\phi}$. Moreover, we have that $\dot{\phi}_k$ is in the forgetful image of $\ddot{\mathcal{F}}_d(\Delta)$ since $x_k, y_{j_k}^{(k)} \in~\Lambda_{\Delta}$.
\end{proof}

\begin{proof}[Proof~of~Theorem~\ref{thm:main}]
Since $\Gamma$ is $Q_\tau$-regular antipodal (and is not virtually cyclic by our assumption that $\Gamma \cap H$ is Zariski-dense in $H$), we have that $\Gamma$ acts minimally on~$\Lambda_\Gamma^\tau$, as is true for an arbitrary nonelementary convergence group acting on its limit set; see~\cite[Theorem~6.13]{zbMATH04021537}. Thus, to prove Theorem~\ref{thm:main}(\ref{denseorbit}), it suffices to find a flag $\varphi \in w(\Lambda_\Gamma^\tau)$ such that $\overline{\Gamma \varphi} \subset G/Q$ contains an entire fiber of the map $w(\Lambda_\Gamma^\tau) \rightarrow \Lambda_\Gamma^\tau$. Let~$\Delta_0$ be a subgroup of $\Gamma \cap H$ that is Zariski-dense and convex cocompact in $H$. We claim that there is in fact some $\varphi \in w(\Lambda_{\Delta_0}^\tau)$ such that $\overline{\Delta_0 \varphi} \subset G/Q$ contains an entire fiber of the map $w(\Lambda_{\Delta_0}^\tau) \rightarrow \Lambda_{\Delta_0}^\tau$. Indeed, we have that $w(\Lambda_{\Delta_0}^\tau) \rightarrow \Lambda_{\Delta_0}^\tau$ is a $\Delta_0$-invariant subbundle of $w(\Lambda_H^\tau) \rightarrow \Lambda_H^\tau$, and the latter may be $H$-equivariantly identified, via some surjective local isomorphism $\iota: H \rightarrow H_d$, where $d=n-1$, with the quotient map $H_d/S \rightarrow H_d/P$, where $S$ is a closed subgroup of $P$ containing $AU$ for some unipotent one-parameter subgroup $U< N$. Setting $\Delta := \iota(\Delta_0)$, it is then enough to find some $\dot{\phi}\in H_d/(AU)$ such that $\overline{\Delta\dot{\phi}} \subset H_d/(AU)$ contains an entire fiber of the quotient map $H_d/(AU) \rightarrow H_d/P$. 
Interpreting the quotient maps $H_d/A \rightarrow H_d/(AU)$ and $H_d/(AU) \rightarrow H_d/P$ as the forgetful maps $\ddot{\mathcal{F}}_d \rightarrow \dot{\mathcal{F}}_d$ and $\dot{\mathcal{F}}_d \rightarrow \mathbb{S}^d$, respectively, we have by a result of Flaminio--Spatzier~\cite[Theorem~1.5]{MR1032882} that there is some $\ddot{\phi} \in \ddot{\mathcal{F}}_d(\Delta)$ such that $\Delta\ddot{\phi}$ is dense in $\ddot{\mathcal{F}}_d(\Delta)$. We then have by Proposition~\ref{prop:spheres} that $\Delta\dot{\phi}$ is dense in $\dot{\mathcal{F}}_d(\Delta)$, where $\dot{\phi}$ is the image of $\ddot{\phi}$ under the forgetful map $\ddot{\mathcal{F}}_d \rightarrow \dot{\mathcal{F}}_d$. Since $\dot{\mathcal{F}}_d(\Delta)$ is the entire preimage of $\Lambda_{\Delta} \subset \mathbb{S}^d$ under the forgetful map $\dot{\mathcal{F}}_d \rightarrow \mathbb{S}^d$, this completes the proof of Theorem~\ref{thm:main}(\ref{denseorbit}). 

To prove Theorem~\ref{thm:main}(\ref{minimality}), we now assume that a point stabilizer for the action of~$H$ on $w(\Lambda_H^\tau)$ contains $A_HN_H$, and we argue that for {\em any} $\varphi \in w(\Lambda_\Gamma^\tau)$, the closure $\overline{\Gamma \varphi} \subset G/Q$ contains an entire fiber of the map $w(\Lambda_\Gamma^\tau) \rightarrow \Lambda_\Gamma^\tau$. Indeed, in this case, we have that $w(\Lambda_\Gamma^\tau)$ is compact, and hence so is $\overline{\Gamma \varphi}$. In particular, the image of $\overline{\Gamma \varphi}$ in $\Lambda_\Gamma^\tau$ is a nonempty closed $\Gamma$-invariant subset of the latter, and hence must be $\Lambda_\Gamma^\tau$ itself by minimality of the action of $\Gamma$ on $\Lambda_\Gamma^\tau$. We deduce that $\overline{\Gamma \varphi}$ has nonempty intersection with $w(\Lambda_{\Delta_0}^\tau)$. On the other hand, since the bundle $w(\Lambda_H^\tau) \rightarrow \Lambda_H^\tau$ may be $H$-equivariantly identified with the bundle $H/S' \rightarrow H/P_H$, where $S'$ is a closed subgroup of $P_H$ containing $A_HN_H$, it follows again from~\cite[Theorem~2]{MR2339285} that the action of $\Delta_0$ on $w(\Lambda_{\Delta_0}^\tau)$ is minimal. We conclude that $\overline{\Gamma \varphi}$ contains $w(\Lambda_{\Delta_0}^\tau)$, i.e., that $\overline{\Gamma \varphi}$ indeed contains all fibers above points in $\Lambda_{\Delta_0}^\tau$. 
\end{proof}

\begin{remark}
The appeal to Proposition~\ref{prop:spheres} in the proof of Theorem~\ref{thm:main}(\ref{denseorbit}) can be avoided if one moreover knows that the limit set $\Lambda_\Delta$ of $\Delta$ in $\mathbb{S}^d$ possesses a point $z_0$ such that any round circle in $\mathbb{S}^d$ through $z_0$ contains at least one other point of $\Lambda_\Delta$. This will be true for instance if~$\Lambda_\Delta$ is a $(d-1)$-dimensional Sierpi\'nski compactum in the sense of Cannon~\cite{MR319203}, or if $\Lambda_\Delta$ is a topological $(d-1)$-sphere (where we are still assuming that $\Delta$ is Zariski-dense in~$H_d$, so that~$\Lambda_\Delta$ cannot be a {\em round} hypersphere of $\mathbb{S}^n$). On the other hand, it was pointed out to the authors by Dongryul Kim that, by an argument similar to the proof of Lemma~2.1 in \cite{MR4956566}, given $z_0 \in \Lambda_\Delta$ and $z_1 \in \mathbb{S}^d \smallsetminus \Lambda_\Delta$, the union of all round circles in $\mathbb{S}^d$ through $z_0$ intersecting $\Lambda_\Delta \smallsetminus \{z_0\}$ has Hausdorff dimension at most $\delta_\Delta+1$, where $\delta_\Delta$ denotes the Hausdorff dimension of $\Lambda_\Delta$. In particular, if $\delta_\Delta < d-1$, then the Hausdorff dimension of such a union is $< d$, so that for every $z_0 \in \Lambda_\Delta$, almost every round circle in $\mathbb{S}^d$ through $z_0$ fails to intersect $\Lambda_\Delta$ in any point apart from~$z_0$. 
\end{remark}

\begin{remark}
We have appealed in the proof of Theorem~\ref{thm:main}, including the proof of Proposition~\ref{prop:spheres}, to a result of Guivarc'h--Raugi~\cite[Theorem~2]{{MR2339285}} to conclude that if $\Delta$ is a Zariski-dense subgroup of $H_d$, then $\Delta$ acts minimally on the preimage in~$H_d/(AN)$ (i.e., the preimage in the orthonormal frame bundle over $\mathbb{S}^d$) of the limit set $\Lambda_\Delta \subset \mathbb{S}^d = H_d/P$. To conclude the latter, one could instead again apply (to any Zariski-dense convex cocompact subgroup of $\Delta$) the ergodicity result of Flaminio--Spatzier~\cite[Theorem~1.5]{MR1032882} and argue as in the proof of \cite[Theorem~3.7]{MR1893917} using conformality of the action of $\Delta$ on $\mathbb{S}^d$. Note further that in the case $d=2$ (which is the relevant case, for instance, in the application to Corollary~\ref{rp3}) one could instead appeal directly to work of Prasad--Rapinchuk~\cite{zbMATH01918959}.
\end{remark}

\section{Proofs of Corollaries~\ref{domainexample}-\ref{rp3} and Proposition~\ref{sl3r}}

\begin{proof}[Proof~of~Corollary~\ref{domainexample}] Up to conjugating $\Gamma$ in $\mathrm{SL}_{n+1}(\mathbb{R})$, we may assume $H = \mathrm{SO}_\circ(n,1)$. We apply Theorem~\ref{thm:main} with $G = \mathrm{SL}_{n+1}(\mathbb{R})$, with $Q=P_{1, \ldots, n}$ a minimal parabolic subgroup of $G$ (i.e., the stabilizer in $G$ of a complete flag of projective subspaces of~$\mathbb{RP}^n$), and with $Q_\tau=P_{1,n}$ the stabilizer of a length-two projective flag consisting of a point in $\mathbb{RP}^n$ and a projective hyperplane in $\mathbb{RP}^n$ containing that point. In this case, the limit set $\Lambda_\Gamma^\tau$ of $\Gamma$ in $G/Q_\tau$ is the set of all length-two projective flags of the form $\{x\} \subset V_{n-1} \subset \mathbb{RP}^n$, where $x \in \partial\Omega$ and $V_{n-1}$ is the projective hyperplane tangent to $\partial \Omega$ at $x$. For $0 \leq i \leq n-1$, we define $w_i \in Q_\tau \backslash G/Q$ as follows: a complete flag of projective subspaces $\{x\} = V_0 \subset V_1 \subset \cdots \subset V_{n-1} \subset \mathbb{RP}^n$, viewed as a point in~$G/Q$, is in position $w_i$ relative to to a projective flag $\{x'\} \subset V'_{n-1} \subset \mathbb{RP}^n$, viewed as a point in~$G/Q_\tau$, where $V'_{n-1} \subset \mathbb{RP}^n$ is a projective hyperplane, if
\begin{itemize}
\item $x' \notin V_j$ for $j < i$;
\item $x' \in V_j$ for $j \geq i$; and
\item $V_{n-1} = V'_{n-1}$.
\end{itemize}
For $i \leq n-2$, a point stabilizer for the action of $H$ on $w_i(\Lambda_H^\tau)$ then contains a conjugate in $H$ of a parabolic subgroup of $\mathrm{SO}_\circ(2,1)$ (for $i=n-2$, this stabilizer is in fact commensurable to a conjugate in $H$ of a proper parabolic subgroup of~$\mathrm{SO}_\circ(2,1)$). This implies Corollary~\ref{domainexample}(\ref{cor:denseorbitdomain}). Moreover, a point stabilizer for the action of $H$ on~$w_0(\Lambda_H^\tau)$ contains (as a finite-index subgroup) a conjugate in $H$ of~$A_H N_H$ in the notation of Theorem~\ref{thm:main}, implying Corollary~\ref{domainexample}(\ref{cor:minimalitydomain}).
\end{proof}

\begin{proof}[Proof~of~Corollary~\ref{sopq}] We may assume $H = \mathrm{SO}_\circ(n,1)$. We apply Theorem~\ref{thm:main} with $G = \mathrm{SO}(n,2)$, with $Q$ a minimal parabolic subgroup of $G$ (i.e., the stabilizer in~$G$ of a complete flag of isotropic subspaces of $\mathbb{R}^{n,2}$), and with $Q_\tau$ the stabilizer of an isotropic line in $\mathbb{R}^{n,2}$. For the purposes of establishing Corollary~\ref{sopq}(\ref{cor:denseorbitsopq}), we define $w \in Q_\tau \backslash G / Q$ as follows: a complete flag of isotropic subspaces $V_1 \subset V_2 \subset \mathbb{R}^{n,2}$, viewed as a point in $G/Q = \mathcal{F}_{\mathrm{SO}(n,2)}$, is in position~$w$ relative to an isotropic line $V'_1 \subset \mathbb{R}^{n,2}$, where the latter is viewed as a point in $G/Q_\tau = \mathrm{Ein}^n $, if $V_1' \subset V_q$. A point stabilizer for the action of $H$ on $w(\Lambda_H^\tau)$ is then some conjugate in $H$ of a proper parabolic subgroup of $\mathrm{SO}_\circ(n-1,1)$. Since every $2$-dimensional isotropic subspace of~$\mathbb{R}^{n,2}$ contains an isotropic line that, viewed as a point in $\mathrm{Ein}^n$, lies in~$\Lambda_\Gamma^\tau$ (see~\cite[Corollary~1.4(2)]{arXiv:2305.15103}), the set $w(\Lambda_\Gamma^\tau)$ is open and dense in $\mathcal{F}_{\mathrm{SO}(n,2)}$, from which we deduce Corollary~\ref{sopq}(\ref{cor:denseorbitsopq}). To establish Corollary~\ref{sopq}(\ref{cor:minimalitysopq}), we instead define $w$ as follows: a complete flag of isotropic subspaces $V_1 \subset V_2 \subset \mathbb{R}^{n,2}$ is in position~$w$ relative to an isotropic line $V'_1 \subset \mathbb{R}^{n,2}$, if~$V_1' = V_1$. In the notation of Theorem~\ref{thm:main}, we have in this case that a point stabilizer for the action of $H$ on $w(\Lambda_H^\tau)$ is some conjugate in~$H$ of $M'_H A_H N_H$, where $M'_H$ is a closed subgroup of $M_H$ isomorphic to $\mathrm{SO}(n-2)$. Corollary~\ref{sopq}(\ref{cor:minimalitysopq}) follows.
 \end{proof}

\begin{proof}[Proof~of~Corollary~\ref{rp3}]
We apply Theorem~\ref{thm:main}(\ref{minimality}) with $G = \mathrm{SL}_4(\mathbb{R})$, $Q = P_1$, and $Q_\tau = P_2$, and with $w \in Q_\tau \backslash G /Q$ defined as follows: a $1$-dimensional linear subspace $V_1 \subset \mathbb{R}^4$, viewed as a point in $\mathbb{RP}^3$, is in position $w$ relative to a $2$-dimensional linear subspace $V_2 \subset \mathbb{R}^4$, where the latter is viewed as a point in $\mathcal{F}_2(\mathbb{R}^4)$, if $V_1 \subset V_2$. In this case, as already observed in the proof of~\cite[Proposition~8.2]{zbMATH07262226}, we have that $w(\Lambda) = \mathbb{RP}^3$ for any antipodal topological $2$-sphere in~$\mathcal{F}_2(\mathbb{R}^4)$. Moreover, the action of $H$ on~$\mathbb{RP}^3$ is transitive with point stabilizers conjugate within $H$ to $A_HN_H$, so Theorem~\ref{thm:main}(\ref{minimality}) indeed applies. 
\end{proof}

\begin{remark}
One can obtain Zariski-dense examples of subgroups $\Gamma < \mathrm{SL}_4(\mathbb{R})$ satisfying the assumptions of Corollary~\ref{rp3} by performing successive ``bendings'' of certain cocompact lattices in~$H$ as in~\cite[Appendix~A]{MR4990468}; that our setting is indeed treated in the latter reference can be seen by identifying $\mathrm{PSL}_4(\mathbb{R})$ with $\mathrm{SO}_\circ(3,3)$.
\end{remark}

\begin{proof}[Proof~of~Proposition~\ref{sl3r}]
Denote by $\Lambda_\Gamma^*$ the limit set of $\Gamma$ in the dual $(\mathbb{RP}^2)^* = \mathcal{F}_2(\mathbb{R}^3)$ of $\mathbb{RP}^2$, and consider the projections $\pi: \mathcal{F}(\mathbb{R}^3) \rightarrow \mathbb{RP}^2$ and $\pi_*: \mathcal{F}(\mathbb{R}^3) \rightarrow (\mathbb{RP}^2)^*$. We claim first that $\pi$ remains surjective when restricted to $\pi_*^{-1}(\Lambda_\Gamma^*)$. Indeed, to say otherwise is to say that $\Lambda_\Gamma^*$ is contained in an affine chart of $(\mathbb{RP}^2)^*$. Suppose for a contradiction that the latter holds. Since $\Lambda_\Gamma^*$ is connected, the convex hull of~$\Lambda_\Gamma^*$ is independent of the choice of affine chart containing $\Lambda_\Gamma^*$ and is thus preserved by~$\Gamma$ (see the proof of Proposition~2.8 in~\cite{zbMATH07262226}), and since $\Lambda_\Gamma^*$ is moreover a topological circle, the interior of this convex hull is nonempty. It follows that every triple of distinct flags in the limit set of $\Gamma$ in $\mathcal{F}(\mathbb{R}^3)$ is positive. On the other hand, since~$\Gamma \cap H$ is Zariski-dense in $H$, the limit set of $\Gamma \cap H$ in $\mathcal{F}(\mathbb{R}^3)$ contains at least three flags, and any triple of distinct flags in the limit set of $\Gamma \cap H$ in $\mathcal{F}(\mathbb{R}^3)$ fails to be positive. 

We conclude that the restriction of $\pi$ to $\pi_*^{-1}(\Lambda_\Gamma^*)$ indeed remains surjective. To complete the proof, it thus suffices to find a dense $\Gamma$-orbit in $\pi_*^{-1}(\Lambda_\Gamma^*)$. To that end, since $\Gamma$ acts minimally on $\Lambda_\Gamma^*$, it in turn suffices to find a flag $\varphi \in \pi_*^{-1}(\Lambda_\Gamma^*)$ such that $\overline{\Gamma \varphi}$ contains an entire fiber of $\pi_\ast$. One can indeed find such $\varphi$ in $\pi_*^{-1}(\Lambda_\Delta^*)$ for any subgroup $\Delta <\Gamma \cap H$ such that $\Delta$ is Zariski-dense and convex cocompact in~$H$, where $\Lambda_\Delta^*$ denotes the limit set of $\Delta$ in $(\mathbb{RP}^2)^*$, as it follows from Hedlund's theorem \cite{zbMATH03023567} that $\Delta$ has a dense orbit in $\pi_*^{-1}(\Lambda_\Delta^*)$.\footnote{More precisely, it follows from Hedlund's theorem that $\Delta \varphi$ is dense in $\pi_*^{-1}(\Lambda_\Delta^*)$ as soon as $\pi(\varphi)$ is neither the unique point in $\mathbb{RP}^2$ fixed by $H$ nor contained in the unique projective line in $\mathbb{RP}^2$ preserved by $H$.}
\end{proof}

\begin{example}
We provide some examples of Zariski-dense Anosov surface groups $\Gamma < \mathrm{SL}_3(\mathbb{R})$ satisfying the assumptions of Proposition~\ref{sl3r}. Consider an abstract surface group
\[
\Gamma_g := \Big\langle a_1, b_1, \ldots, a_g, b_g \ \big | \ [a_1, b_1]\cdots[a_g,b_g] \Big\rangle
\]
of genus $g \geq 2$, and let $\rho: \Gamma_g \rightarrow \mathrm{SL}_2(\mathbb{R})$ be a discrete and faithful representation. Given a $g$-tuple $v = (v_1, \ldots, v_g)$ of vectors $v_i \in \mathbb{R}^2$ and a $g$-tuple $\omega = (\omega_1, \ldots, \omega_g)$ of letters $\omega_i \in \{C,R\}$, we define a deformation $\rho_v^\omega : \Gamma_g \rightarrow \mathrm{SL}_3(\mathbb{R})$ as follows. For $i=1, \ldots, g$, in the case that $\omega_i = C$, we set
\[
\rho_v^\omega(a_i) = \begin{pmatrix} \rho(a_i) & v_i \\ & 1 \end{pmatrix}, \quad \rho_v^\omega(b_i) = \begin{pmatrix} \rho(b_i) & u_i \\ & 1 \end{pmatrix},
\]
where $u_i \in \mathbb{R}^2$ is the unique vector satisfying
\[
\left[ \begin{pmatrix} \rho(a_i) & v_i \\ & 1 \end{pmatrix}, \begin{pmatrix} \rho(b_i) & u_i \\ & 1 \end{pmatrix}  \right] = \begin{pmatrix} [\rho(a_i),\rho(b_i)] & \\ & 1 \end{pmatrix}.
\]
In the case that $\omega_i = R$, we set
\[
\rho_v^\omega(a_i) = \begin{pmatrix} \rho(a_i) & \\ v_i^T & 1 \end{pmatrix}, \quad \rho_v^\omega(b_i) = \begin{pmatrix} \rho(b_i) & \\ w_i^T &  1 \end{pmatrix},
\]
where $w_i \in \mathbb{R}^2$ is again the unique vector satisfying
\[
\left[ \begin{pmatrix} \rho(a_i) & \\ v_i^T & 1 \end{pmatrix}, \begin{pmatrix} \rho(b_i) & \\ w_i^T & 1 \end{pmatrix}  \right] = \begin{pmatrix} [\rho(a_i),\rho(b_i)] & \\ & 1 \end{pmatrix}.
\]
If each of the $v_i$ is sufficiently close to the origin in $\mathbb{R}^2$, then the representation $\rho_v^\omega$ is Anosov. If $v_i=0$ for some $i$, then $\rho_v^\omega(\Gamma_g) \cap H$ is Zariski-dense in $H$, where $H$ is the upper block-diagonal embedding of $\mathrm{SL}_2(\mathbb{R})$ in $\mathrm{SL}_3(\mathbb{R})$. If there are in addition two distinct indices $j,k$ such that $\omega_j = C$, $\omega_k = R$, and each of $v_j$ and $v_k$ is nonzero, then $\rho_v^\omega(\Gamma_g)$ is Zariski-dense in $\mathrm{SL}_3(\mathbb{R})$; indeed, in this case, we have that the Lie algebra $\mathfrak{g}$ of the Zariski-closure $G$ of $\rho_v^\omega(\Gamma_g)$ contains two matrices of the form
\[
\begin{pmatrix} \ast & x \\ 0 & 0\end{pmatrix}, \quad \begin{pmatrix} \ast & 0 \\ y^T & 0\end{pmatrix},
\]
where $x,y \in \mathbb{R}^2$ are both nonzero. On the other hand, since $\mathfrak{g}$ also contains the Lie algebra $\mathfrak{h}$ of $H$, we have that $\mathfrak{g}$ indeed contains the matrices
\[
\begin{pmatrix} 0 & x \\ 0 & 0\end{pmatrix}, \quad \begin{pmatrix} 0 & 0 \\ y^T & 0\end{pmatrix},
\]
so that $G$ nontrivially intersects each of the horospherical subgroups of $\mathrm{SL}_3(\mathbb{R})$ normalized by $H$. Since $\mathrm{SL}_2(\mathbb{R})$ acts transitively on $\mathbb{R}^2-\{0\}$, it follows that $G$ in fact contains each of those horospherical subgroups, and since any pair of opposite horospherical subgroups together generate $\mathrm{SL}_3(\mathbb{R})$, we conclude that $G = \mathrm{SL}_3(\mathbb{R})$.
\end{example}

\section{Limit sets that do not fiber}
In this section, we prove Theorem~\ref{distinctfibers}, that is, we give examples of Zariski-dense groups that are Anosov in a partial flag manifold but such that the natural projection from the Benoist--Guivarc'h limit set in the Furstenberg boundary to the Anosov limit set is not a fibration.

\begin{proof}[Proof~of~Theorem~\ref{distinctfibers}] For $1\leq i \leq 2n-1$, set $\ell_i:=\binom{2n}{i}$ and fix the basis for the exterior power $\bigwedge^i \mathbb{R}^{2n} \simeq \mathbb{R}^{\ell_i}$ consisting of the vectors \[\mathcal{V}_i:=\{e_{m_1}\wedge \cdots \wedge e_{m_i}: m_1<\cdots <m_i\leq 2n\}.\] 

For the rest of the proof, we equip $\mathbb{R}^{\ell_i}$ with the unique inner product with respect to which the basis $\mathcal{V}_i$ is orthonormal; for $v\in \mathbb{R}^{\ell_i}$ non-zero, we will denote by $v^{\perp}$ the orthogonal complement of $v$ with respect to the latter inner product. We also equip $\mathbb{P}(\mathbb{R}^{\ell_i})$ with a Riemannian metric and denote by $\mathcal{B}_{\epsilon}([v])$ the ball of radius $\epsilon>0$ centered at $[v]\in \mathbb{P}(\mathbb{R}^{\ell_i})$.

For $1\leq i \leq 2n-1$, define the rank-two matrices $J_n,L_n\in \textup{Mat}(\mathbb{R}^{\ell_n})$ as follows:\begin{align*}J_n(e_1\wedge \cdots \wedge e_{n})&= e_1\wedge \cdots \wedge e_{n},\\ J_n(e_1\wedge \cdots \wedge e_{n-1}\wedge e_{n+1})&=e_1\wedge \cdots \wedge e_{n-1}\wedge e_{n+1},\\ L_n(e_{2n}\wedge \cdots \wedge  e_{n+1})&=e_{2n}\wedge \cdots \wedge e_{n+1},\\ L_n(e_{2n}\wedge \cdots \wedge e_{n+2}\wedge e_n)&=e_{2n}\wedge \cdots \wedge e_{n+2}\wedge e_{n},\\ J_n v=0, \ v\in \mathcal{V}_n-  \{e_1\wedge \cdots& \wedge e_n, e_1\wedge \cdots\wedge e_{n-1}\wedge e_n\},\\ L_n v'=0, \ v'\in \mathcal{V}_n-  \{e_{2n}\wedge \cdots &\wedge e_n, e_{2n}\wedge \cdots\wedge e_{n+2}\wedge e_{n-1}\}.\end{align*}

We will need the following observation.

\begin{fact}\label{choice-h} There is $h\in \mathrm{GL}_{2n}(\mathbb{R})$ simultaneously \hbox{satisfying the following conditions:}
\begin{enumerate}[label=(C\arabic*)]
\item \label{it-1} $(\wedge^n  h) J_n (\wedge^n h^{-1})(e_1\wedge \cdots \wedge e_n) \notin (e_1\wedge\cdots \wedge e_n)^{\perp};$
\item \label{it-2} $(\wedge^n  h) J_n (\wedge^n h^{-1})(e_1\wedge \cdots \wedge e_n) \notin (e_n\wedge\cdots \wedge e_{2n})^{\perp};$
\item \label{it-3} $(\wedge^n  h) L_n (\wedge^n h^{-1})(e_1\wedge \cdots \wedge e_n) \notin (e_1\wedge\cdots \wedge e_n)^{\perp};$
\item \label{it-4}  $(\wedge^i  h) L_n (\wedge^i h^{-1})(e_1\wedge \cdots \wedge e_i) \notin (e_n\wedge\cdots \wedge e_{2n})^{\perp};$
\item \label{it-5} $h^{\pm 1}e_i\notin e_1^{\perp}\cup \cdots \cup  e_n^{\perp}$ for any $i=1,\ldots,2n$;
\item \label{it-6} for any $i=1,\ldots,2n-1$, each of the $i$-planes \[ h^{\pm 1}\operatorname{span}\{e_1,\ldots, e_i\}, \> h^{\pm 1}\operatorname{span}\{e_{2n},\ldots, e_{2n-i+1}\}\] is transverse to each of the $(2n-i)$-planes \[\operatorname{span}\{e_1,\ldots, e_{2n-i}\}, \> \operatorname{span}\{e_{2n},\ldots, e_{i+1}\}.\]
\end{enumerate}
\end{fact}
\begin{proof} 
It suffices to verify that each of the Zariski-open conditions \ref{it-1}-\ref{it-6} defines a nonempty subset of $\mathrm{GL}_{2n}(\mathbb{R})$. Indeed, each of these sets is then open and dense in $\mathrm{GL}_{2n}(\mathbb{R})$, so that their intersection is nonempty. 
Note that \ref{it-1} and~\ref{it-3} trivially hold when $h=\textup{I}_{2n}$, and that \ref{it-5} and \ref{it-6} clearly define nonempty subsets of $\textup{GL}_{2n}(\mathbb{R})$. To see that \ref{it-3} holds, let $h_0\in \textup{GL}_{2n}(\mathbb{R})$ be the unipotent matrix satisfying $h_0e_{i}=e_i$ when $i\geq n+1$ and $h_0e_i=e_i+e_{n+i}$ for $1\leq i \leq n$. Then \begin{align*}(\wedge^{n}h_0)J_n(\wedge^n h_0^{-1})(e_1\wedge \cdots \wedge e_{n})&=(\wedge^n h_0)J_n((e_1-e_{n+1})\wedge \cdots \wedge (e_{n}-e_{2n}))\\ &=(\wedge^n h_0)(e_1\wedge \cdots \wedge e_n)\\ &=(e_1+e_{n+1})\wedge \cdots \wedge (e_{n}+e_{2n}),\end{align*} and the latter vector is not in $(e_{2n}\wedge \cdots \wedge e_{n+1})^{\perp}$. Moreover, since $L_n$ and $J_n$ are conjugate via the matrix $\wedge^n w_0$, where $w_0e_i=e_{2n-i+1}$ for $1\leq i \leq 2n$, we have that~\ref{it-4} also is nonempty.\end{proof}

After choosing $h\in \mathrm{GL}_{2n}(\mathbb{R})$ satisfying conditions \ref{it-1}-\ref{it-6} in Fact \ref{choice-h}, we choose a matrix $$B:=h \begin{pmatrix} \textup{diag}(\mu_1,\ldots,\mu_{n-1})  & &  \\ & 1 & &  \\ & & 1 &  \\ & & & \textup{diag}\left(\frac{1}{\mu_{n-1}},\ldots,\frac{1}{\mu_1}\right) \end{pmatrix}h^{-1}$$ where $\mu_1>\cdots>\mu_n>1$ are chosen such that, for any $m\geq 1$, the Zariski-closure of the cyclic group generated by $B^m$ in $\textup{SL}_{2n}(\mathbb{R})$ is the torus \begin{align}\label{torus}\mathbb{T}:=h\big\{\textup{diag}(\theta_1,\ldots,\theta_{n-1},1,1,\theta_{n-1}^{-1},\ldots,\theta_{n}^{-1}):\theta_i\in \mathbb{R}^{\ast}\big\}h^{-1}.\end{align}

Choose also a loxodromic matrix $$A:=\begin{pmatrix} \lambda_1 & & \\ &  \ddots &  \\ & &  \lambda_{2n}\end{pmatrix}$$ where $\lambda_1>\cdots >\lambda_{2n}>0$ and $\lambda_1\cdots \lambda_{2n}=1$, and the $\lambda_i$ are chosen such that $A$ does not lie in any conjugate of either $\textup{Sp}_{2n}(\mathbb{C})$ or $\mathrm{SO}_{2n}(\mathbb{C})$ in $\mathrm{GL}_{2n}(\mathbb{C})$.

For every $i$, we have that  $\wedge^i A\in \mathrm{SL}_{\ell_i}(\mathbb{R})$ is biproximal and its attracting and repelling fixed flags in $\mathcal{F}_{1,\ell_i-1}(\mathbb{R}^{\ell_i})$ are \begin{align*}\operatorname{span}\{e_1\wedge \cdots \wedge e_i\}&\subset  (e_{2n}\wedge \cdots \wedge e_{2n-i+1})^{\perp},\\ \operatorname{span}\{e_{2n}\wedge \cdots \wedge e_{2n-i+1}\}&\subset  (e_1\wedge \cdots \wedge e_i)^{\perp}.\end{align*}

Note also that $\wedge^j B$ is biproximal for every $j=1,\ldots,n-1$. For any such $j$, by the choice of the matrix $h$ satisfying \ref{it-6}, the limit sets of the matrices $\wedge^j A$ and~$\wedge^j B$ are transverse (i.e., antipodal) in $\mathcal{F}_{1,\ell_j-1}(\mathbb{R}^{\ell_j})$, thus, by \cite{MR4002289}, there is $m_0>0$ such that for any $m\geq m_0$, the subgroup $\langle A^m,B^m\rangle < \mathrm{SL}_{2n}(\mathbb{R})$ is a rank-two free group and is $P_{1,\ldots, n-1}$-Anosov in $\mathrm{SL}_{2n}(\mathbb{R})$.

By the definition of the matrix $B$, in $\mathbb{P}(\textup{Mat}(\mathbb{R}^{\ell_n}))$ we have\begin{align*}\lim_{r\rightarrow \infty} [\wedge^n B^r]&=(\wedge^nh)J_n(\wedge^n h^{-1}),\\ \lim_{r\rightarrow \infty} [\wedge^n B^{-r}]&=(\wedge^nh)L_n(\wedge^n h^{-1}).\end{align*}

For $\varepsilon>0$ and $1\leq i \leq 2n-1$, define $$\mathcal{S}_{\varepsilon}^i:=\big\{[v]\in \mathbb{P}\left(\mathbb{R}^{\ell_i}\right):\textup{dist}([v],\mathbb{P}((e_1\wedge \cdots \wedge e_{i})^{\perp}\cup (e_{2n}\wedge \cdots \wedge e_{2n-i+1})^{\perp})\geq \varepsilon\big\}.$$
Note that there is $\varepsilon_0>0$ such that the set 
\begin{align}\label{M-i}\mathcal{M}_{\varepsilon}^{i}:=\Big\{\big(\operatorname{span}\{v_1\} \subset \cdots \subset \operatorname{span}\{v_1,\ldots,v_{2n-1}\}\big): [v_1\wedge \cdots \wedge v_i]\in \mathcal{S}_{\varepsilon}^i\Big\}\end{align} has non-empty interior for any $0<\varepsilon\leq \varepsilon_0$ and any $1\leq i \leq 2n-1$.

Using the previous relations, conditions \ref{it-1}-\ref{it-6} show that there is $0<\varepsilon\leq \varepsilon_0$ and $m_1\geq m_0$ such that for any $q\in \mathbb{Z}^{\ast}$, $m\geq m_1$, and $i=1,\ldots, 2n-1$ \begin{align}\label{pp3} (\wedge ^i B^{qm})\big(\mathcal{B}_{\varepsilon}([e_1\wedge \cdots \wedge e_i])\cup \mathcal{B}_{\varepsilon}([e_{2n}\wedge \cdots \wedge e_{2n-i+1}])\big)\subset \mathcal{S}_{\varepsilon}^{i},&\\ \label{pp4}( \wedge^i A^{qm})\mathcal{S}_{\varepsilon}^i\subset \mathcal{B}_{\varepsilon}([e_1\wedge \cdots\wedge e_i])\cup \mathcal{B}_{\varepsilon}([e_{2n}\wedge \cdots \wedge e_{2n-i+1}])& \end{align}

Now fix $m\geq m_1$ and let $\Gamma_m:=\langle A^m,B^m\rangle$. Recall by the choice of $m_1\geq m_0$ that~$\Gamma_m$ is $P_{1, \ldots, n-1}$-Anosov in $\textup{SL}_{2n}(\mathbb{R})$. We will need to verify the following claims. 

\begin{claim} The group $\Gamma_m$ is Zariski-dense in $\mathrm{SL}_{2n}(\mathbb{R})$.\end{claim}

\begin{proof} Let $G$ (respectively,  $G_{\mathbb{C}}$) be the Zariski-closure of $\Gamma_m$ in $\textup{SL}_{2n}(\mathbb{R})$ (resp., in~$\textup{SL}_{2n}(\mathbb{C})$). Note that since $B^m\in G$, by the definition of $B$, the torus $\mathbb{T}$ (see (\ref{torus})) is contained in $G$, as is the matrix $$B_{t}:=h\begin{pmatrix} t & &\\ & \textup{I}_{2n-2} & \\ & & \frac{1}{t}\end{pmatrix}h^{-1}$$
for every $t\neq0$.

We first check that $\Gamma_m$ is an irreducible subgroup of $\mathrm{SL}_{2n}(\mathbb{C})$. 
 Indeed, suppose $V \subset \mathbb{C}^{2n}$ is a nontrivial $\Gamma_m$-invariant subspace. Then since $V$ is in particular $A^{qm}$-invariant for any $q$, we have that  $V=\operatorname{span}\{e_i:i\in I\}$ where $\textup{dim}_{\mathbb{C}}V=|I|$.
Suppose that $I$ is a proper subset of $\{1,\ldots,2n\}$. 

Denote by $\langle\cdot,\cdot\rangle$ the standard Hermitian inner product on $\mathbb{C}^{2n}$. Let $r\in \{1,\ldots,n\}\smallsetminus I$ and $s\in I$. Since $V$ is $G$-invariant and is a subspace of $e_r^{\perp}$, we have $\langle B_t e_s,e_r\rangle=0$. Note that $B_t=\textup{I}_{2n}+ (t-1) Q_t$, where $$Q_t:=h\begin{pmatrix} 1 & &\\ & 0_{2n-2} &\\ & & -\frac{1}{t} \end{pmatrix}h^{-1}.$$ Therefore, $\langle Q_te_s,e_r\rangle =0$. A straightforward computation gives that $$\langle Q_te_s,e_r\rangle=\langle h^{-1}e_s,e_1\rangle \cdot \langle he_1,e_r\rangle-\frac{1}{t}\langle h^{-1}e_s,e_n\rangle \cdot \langle he_n,e_r\rangle.$$ The latter expression however cannot be identically zero, since, by \ref{it-5}, all entries of $h$ and $h^{-1}$ are nonzero, a contradiction. Hence $V=\mathbb{C}^n$ and $\Gamma_m$ is irreducible.

Since sufficiently large powers of $A$ and $B$ are contained in a finite-index subgroup of $\Gamma_m$, the above argument shows that any finite-index subgroup  of $\Gamma_m$ acts irreducibly on $\mathbb{C}^{2n}$. In particular, (e.g., see \cite[Lemma 5.1]{ADLM}) the Zariski-closure of~$\Gamma_m$ in $\mathrm{SL}_n(\mathbb{C})$ is semisimple. As $B^m \in G_{\mathbb{C}}$, by \cite[Theorem 4.1]{ADLM}, we have that~$G_{\mathbb{C}}$ is either $\mathrm{SL}_{2n}(\mathbb{C})$ or is conjugate to either $\mathrm{Sp}_{2n}(\mathbb{C})$ or $\mathrm{SO}_{2n}(\mathbb{C})$. By our choice of the eigenvalues of $A$, we deduce that $G_{\mathbb{C}}=\mathrm{SL}_{2n}(\mathbb{C})$. This shows that $\Gamma_m$ is Zariski-dense in $\mathrm{SL}_{2n}(\mathbb{R})$.\end{proof}

Now let $\pi: \Lambda_{\Gamma_m}^{\mathcal{F}}\rightarrow \Lambda_{\Gamma_m}^{1,2n-1}$ be the $\Gamma_m$-equivariant projection.

\begin{claim} Let $a^{+}$ be the attracting fixed point of $A$ in $\mathcal{F}_{1,2n-1}(\mathbb{R}^{2n})$. Then the sole element of $\pi^{-1}(a^{+})$ is the attracting fixed point of $A$ in $\mathcal{F}(\mathbb{R}^{2n})$.\end{claim}

\begin{proof} Fix $x\in \pi^{-1}(a^{+})$. Take an infinite sequence $(g_r)_r$ of elements in $\Gamma_m$ which are loxodromic and such that $\lim_{r}x_{g_r}^{+}=x$, where $x_{g_r}^{+}$ denotes the attracting fixed flag of $g_r$ in $\mathcal{F}(\mathbb{R}^{2n})$. The attracting fixed points of $g_r$ in $\mathbb{P}(\mathbb{R}^{2n})$ converge to the line $[e_1]$. Hence,  since $\Gamma$ is $P_1$-Anosov, there is an infinite sequence of integers $(k_{1r})$ with $\lim_r k_{1r}=\infty$ such that $g_r\in \langle A^m,B^m\rangle$ has reduced form $$g_r=A^{mk_{1r}}B^{ms_{1r}}\cdots A^{mk_{\ell_r r}}B^{ms_{\ell_r r}}.$$
Indeed, to see the previous claim, note that if $(k_{1r})_r$ had a bounded subsequence, then there would be a subsequence $(g_{r_N})_N$ of $(g_r)_r$ and $p\in \mathbb{Z}$ such that $A^{-pm}g_{r_N}A^{pm}$ starts with a nontrivial power of $B^m$. However, by (\ref{pp3}) and (\ref{pp4}) for $i=1$, it follows that for any word of the form $w\in \langle A^m,B^m\rangle$ starting and ending with a nontrivial power of $B^m$, we have $[w^{\pm 1}e_1],[w^{\pm 1}e_{2n}]\in \mathcal{S}_{\varepsilon}^1$.
This shows that for any infinite sequence $(w_r)$ starting and ending with a power of $B^m$, the limit of attracting fixed points of any sequence of the form $(w_rA^{q_r m})$ in $\mathbb{P}(\mathbb{R}^{2n})$ is equal to the limit of attracting fixed points of $(w_n)$ which remains in $\mathcal{S}_{\varepsilon}^1$. This indeed shows that $\lim k_{1r}=\infty$.

Up to passing to a subsequence, up to replacing $x$ with $A^m x$, and up to replacing~$g_r$ with $A^mg_rA^{-m}$, we may assume that $g_r$ ends with a power of $A^m$, i.e., that $k_{\ell_r r}\neq 0$, $s_{\ell_r r}=0$, and  $$g_r=A^{mk_{1r}}B^{ms_{1r}}\cdots A^{mk_{\ell_r r}}.$$

Without loss of generality, we may also assume that $s_{1r}\neq 0$, since otherwise $g_r$ is a power of $A^m$ and it follows readily that $x$ is the attracting fixed point of $A$ in~$\mathcal{F}(\mathbb{R}^{2n})$. Up to passing to a further subsequence of $(g_r)_r$, we may assume that the sequence of repelling fixed flags of $(g_r)_r$ converges to a flag $F_0\in \mathcal{F}(\mathbb{R}^{2n})$. Choose a flag $x_0\in \mathcal{F}(\mathbb{R}^{2n})\cap \bigcap_{i=1}^{2n-1}\mathcal{M}_{\varepsilon}^i$ (see (\ref{M-i})), $$x_0:=\big(\operatorname{span}\{v_1\} \subset \operatorname{span}\{v_1,v_2\}\subset \cdots \subset \operatorname{span}\{v_1,\ldots,v_{2n-1}\}\big),$$ which is antipodal to $F_0$ and such that $[v_1\wedge \cdots \wedge  v_i]\in \mathcal{S}_{\varepsilon}^i$ for $i=1,\ldots,2n-1$. Such $x_0$ exists since for  fixed $\varepsilon>0$ the set $\bigcap_{i=1}^{2n-1}\mathcal{M}_{\varepsilon}^i$ has nonempty interior. As~$x_0$ is transverse to $F_0$, we have $\lim_r g_rx_0=\lim_r x_{g_r}^{+}=x$.

By using the ping-pong inclusions (\ref{pp3}) and (\ref{pp4}), for any $p,q\neq 0$, $i=1,\ldots,2n-1$, \begin{align*}\wedge^i (B^{pm}A^{qm})[v_1\wedge\cdots \wedge v_i]\in \wedge^i B^{pm}(\mathcal{B}_{\varepsilon}([e_1\wedge \cdots \wedge e_i])\cup \mathcal{B}_{\varepsilon}([e_{2n}\wedge \cdots \wedge e_{2n-i+1}]))&\subset \mathcal{S}_{\varepsilon}^i,\\ \wedge^i A^{qm}\mathcal{S}_{\varepsilon}^i\subset \mathcal{B}_{\varepsilon}([e_1\wedge\cdots \wedge e_i])\cup \mathcal{B}_{\varepsilon}([e_{2n}\wedge\cdots \wedge e_{2n-i+1}]).\end{align*} The previous inclusions imply $$(\wedge^i g_r)[v_1\wedge \cdots \wedge v_i]\in(\wedge^i A^{mk_{1r}})\mathcal{S}_{\varepsilon}^i.$$ Therefore, as $\lim_r k_{1r}=\infty$ and $(\wedge^i A^{mk_{1r}})\mathcal{S}_{\varepsilon}^i$ Hausdorff converges to $\{[e_1\wedge\cdots \wedge e_i]\}$, we have $$\lim_{r\rightarrow \infty} \wedge^i g_r [v_1\wedge\cdots  \wedge v_i]=[e_1\wedge\cdots \wedge e_i].$$ 
By the choice of $x_0\in \mathcal{F}(\mathbb{R}^{2n})$, as $\lim_r g_rx_0=x$, we have $\lim_rg_r\operatorname{span}\{v_1,\ldots, v_i\}=\operatorname{span}\{e_1,\ldots,e_i\}$ for every $1\leq i \leq 2n-1$, and thus $$x=\big(\operatorname{span}\{e_1\}\subset \operatorname{span}\{e_1,e_2\}\subset\cdots \subset  \operatorname{span}\{e_1,\ldots,e_{2n-1}\}\big).$$ This completes the proof of the claim.\end{proof}

The following claim follows from~\cite[Proposition~9.5]{zbMATH06882286}, but we include a proof for the convenience of the reader.

\begin{claim} Let $b^{+}=(\operatorname{span}\{he_1\}\subset \operatorname{span}\{he_1,\ldots,he_{2n-1}\})$ be the attracting fixed point of $B^m$ in $\mathcal{F}_{1,2n-1}(\mathbb{R}^{2n})$. Then  $\pi^{-1}(b^{+})$ is infinite.\end{claim}

\begin{proof} Let $A^{+}$ (respectively, $A^{-}$) be the attracting (resp., repelling) fixed points of~$A^m$ in $\mathcal{F}(\mathbb{R}^{2n})$. Let $\mathcal{S}$ be the subset of $\Gamma_m$ consisting of all $g\in \Gamma_m$ such that $h^{-1}gA^{+}$ is antipodal to $A^{+}$ and $A^{-}$.

Set $V_1:=\textup{Ker}J_n$ and $$V_2 :=\big\{e_1\wedge \cdots \wedge e_{n-1}\wedge (xe_{n}+ye_{n+1}):x,y\in \mathbb{R}\big\}\simeq \mathbb{R}e_{n}\oplus \mathbb{R}e_{n+1}.$$ We have the decomposition $\bigwedge^n \mathbb{R}^{2n}=V_1\oplus V_2$ and recall that $J_n:\bigwedge^n \mathbb{R}^{2n}\rightarrow V_2$ is the natural projection onto $V_2$. 

Since $\mathcal{S}$ is Zariski-dense in $\textup{SL}_{2n}(\mathbb{R})$, the set $$\big\{h^{-1}ge_1\wedge \cdots \wedge h^{-1}ge_{n}:g\in \mathcal{A}\big\}$$ has infinite projection on $V_2$. Indeed, if the latter set had finite projection on $\mathbb{P}(V_2)$, there would exist nonzero $(x_1,y_1),\ldots,(x_{q},y_q)\in \mathbb{R}^2$ such that $\mathcal{S}$ is contained in $$\bigcup_{i=1}^{q}\Big\{g \in \textup{SL}_{2n}(\mathbb{R}):  h^{-1}ge_1\wedge \cdots \wedge h^{-1}ge_n\in (e_1\wedge \cdots \wedge e_{n-1}\wedge (x_ie_{n}+y_ie_{n+1}))^{\perp}\Big\},$$ 
but the latter union is a proper algebraic subset of $\mathrm{SL}_{2n}(\mathbb{R})$.

It follows that there is an infinite sequence $(w_r)_r$ in $\mathcal{S}$ such that $h^{-1}w_rA^{+}$ is transverse to $A^{\pm}$ and the collection of projections $\{[\omega_r]\}_r\subset \mathbb{P}(\mathbb{R}e_n\oplus \mathbb{R}e_{n+1})$, $\omega_r\neq 0$, of the $h^{-1}w_re_1\wedge \cdots \wedge h^{-1}w_re_n$ onto $\mathbb{P}(V_2)\simeq \mathbb{P}(\mathbb{R}e_{n}\oplus \mathbb{R}e_{n+1})$ is infinite.

For every $r$, since $h^{-1}w_rA^{+}\in \Lambda_{\Gamma_m}^{\mathcal{F}}$ is antipodal to $A^{-}$, we have for any $j\neq n$, $$\lim_{p\rightarrow \infty}(\wedge^j B^{mp})[w_re_1\wedge \cdots \wedge w_re_j]=[he_1\wedge \cdots \wedge he_j].$$  Since $\lim_{p}[\wedge^n B^{mp}]=(\wedge^n h)J_n(\wedge^n h^{-1})$, for every $r\in \mathbb{N}$, it follows that $$\lim_{p\rightarrow \infty}[\wedge^n B^{mp}(w_re_1\wedge \cdots \wedge w_re_n)]=(\wedge^n h)[e_1\wedge\cdots \wedge e_{n-1}\wedge \omega_r].$$ In other words, it follows that $\pi^{-1}(b^+)$ contains the infinite set of flags \begin{align*}\Big\{\big(\operatorname{span}\{he_1\}\subset \cdots &\subset \operatorname{span}\{he_1,\ldots,he_{n-1},h\omega_r\}\subset\cdots \subset \operatorname{span}\{he_1,\ldots,he_{2n-1}\}\big)\Big\}_r,\end{align*}
and the claim is proved.\end{proof}

We conclude that the rank-two free subgroup $\Gamma:= \Gamma_m$ of $\textup{SL}_{2n}(\mathbb{R})$ is as desired in Theorem~\ref{distinctfibers}.\end{proof}

\begin{remark}
We remark that the properties of the limit set $\Lambda_\Gamma^\mathcal{F}$ of the examples~$\Gamma$ in the proof of Theorem~\ref{distinctfibers} are somewhat delicate. 
For example, for $n=2$, it is not difficult to see that the subgroup $\Gamma < \mathrm{SL}_4(\mathbb{R})$ constructed in the proof of Theorem~\ref{distinctfibers} can be approximated by subgroups $\Gamma' < \mathrm{SL}_4(\mathbb{R})$ such that $\Lambda_{\Gamma'}^\mathcal{F}$ is the entire preimage of~$\Lambda_{\Gamma'}^{1,3}$ under the projection $\mathcal{F}(\mathbb{R}^4) \rightarrow \mathcal{F}_{1,3}(\mathbb{R}^d)$. Moreover, one can also approximate~$\Gamma$ by $P_{1,2}$-Anosov subgroups $\Gamma'' < \mathrm{SL}_4(\mathbb{R})$ as in~\cite[Corollary~5.5]{arXiv:2603.25098}, for which the $\Gamma''$-equivariant projection $\Lambda_{\Gamma''}^\mathcal{F} \rightarrow \Lambda_{\Gamma''}^{1,3}$ is injective. 
\end{remark}

\subsection*{Acknowledgements} We thank Dongryul Kim, Fran\c{c}ois Labourie, Giuseppe Martone, Yosuke Morita, and Franco Vargas Pallete for helpful discussions.
We also thank Joaqu\'in Lema for suggesting to us the question as to whether Anosov groups as in Theorem~\ref{distinctfibers} exist. This work is based on discussions held: during a visit of the first-named author to the Institut des Hautes \'Etudes Scientifiques in March 2025; while all three authors were in residence at the Simons Laufer Mathematical Sciences Institute in Berkeley, California, during the Spring 2026 semester (National Science Foundation Grant No. DMS2424139); and during a visit of the authors to the Lodha Mathematical Sciences Institute in August-September 2026. We thank all three of these institutions for their hospitality. 

The mathematical content presented here was obtained entirely by the authors without the use of AI, nor was any form of AI used in the preparation of this manuscript.
\bibliography{minimalbib}{}
\bibliographystyle{siam}

\end{document}